\documentclass[11pt]{article}
\usepackage[utf8]{inputenc}
\usepackage[english]{babel}
\usepackage{amssymb}
\usepackage{xcolor}
\usepackage{amsthm}
\usepackage{amsmath}
\usepackage{mathtools}
\usepackage{geometry}
\usepackage{cancel}
\usepackage[colorinlistoftodos,prependcaption,textsize=tiny]{todonotes}
\title{Long QMDS additive code}
\author{Daniele Bartoli\thanks{Dipartimento di Matematica e Informatica, Universit\`a degli Studi di Perugia,  Perugia, Italy. daniele.bartoli@unipg.it},  
Alessandro Giannoni\thanks{Dipartimento di Matematica e Applicazioni  ``R. Caccioppoli'', Università di Napoli Federico II, Napoli, Italy, alessandro.giannoni@unina.it},
Giuseppe Marino\thanks{Dipartimento di Matematica e Applicazioni  ``R. Caccioppoli'', Università di Napoli Federico II, Napoli, Italy, giuseppe.marino@unina.it},
Yue Zhou\thanks{Department of Mathematics, National University of Defense Technology, 410073 Changsha, China, yue.zhou.ovgu@gmail.com}}

\date{}

\newtheorem{theorem}{Theorem}[section]
\newtheorem{lemma}[theorem]{Lemma}

\newtheorem{remark}[theorem]{Remark}
\newtheorem{cor}[theorem]{Corollary}
\newtheorem{prop}[theorem]{Proposition}
\newtheorem{definition}[theorem]{Definition}

\def\F{\mathbb{F}}

\begin{document}

\maketitle
\begin{abstract}
We investigate additive codes, defined as $\mathbb{F}_q$-linear subspaces $C \subseteq \mathbb{F}_{q^h}^n$ of length $n$ and dimension $r$ over $\mathbb{F}_q$. An additive code is said to be of type $[n, r/h, d]_q^h$, where $d$ denotes the minimum Hamming distance and the normalized dimension $r/h$ may be fractional. A central object of interest is the class of quasi-MDS (QMDS) codes, those additive codes achieving the generalized Singleton bound:

$$
d = n - \left\lceil \frac{r}{h} \right\rceil + 1.
$$

In this work, we construct explicit families of additive QMDS codes whose lengths exceed those of the best-known $\mathbb{F}_{q^h}$-linear MDS codes which is $q^h+1$, and we will call these types of codes ``Long''
. By leveraging $\mathbb{F}_q$-linearity and geometric tools like partial spreads and dimensional dual arcs, we show that additive structures allow longer codes without sacrificing optimality in distance. We also examine dual codes and give conditions under which the QMDS property is preserved under duality.

\end{abstract}
\section{Introduction}

Additive codes generalize classical linear codes by relaxing the requirement of linearity over the ambient alphabet. This broader framework captures a wide range of code families, including subfield-linear and trace-orthogonal codes, and has proven especially relevant in applications such as byte error correction~\cite{raviv2023quantum}, quantum coding, and low-density constructions~\cite{boucher2019ldpc, zhang2022ldpc}.

A central class of interest within  coding theory is that of maximum distance separable (MDS) codes, which achieve the Singleton bound
$$d\leq n-k+1.$$ However, in the additive setting, the Singleton bound depends not only on the code length and dimension, but also on the field extension degree $h$. In this context, the appropriate generalization is that of quasi-MDS (QMDS) codes, which attain the bound
\[
d = n - \left\lceil \frac{r}{h} \right\rceil + 1,
\]
where $r$ denotes the dimension of the code over $\mathbb{F}_q$~\cite{martinezpenas2024linear}. When both a code and its dual satisfy this equality, we say the code is \emph{dually QMDS}.

Recent work has shown that additive QMDS codes offer more flexibility than their classical linear counterparts. In particular, they may exceed the maximum known length of $\mathbb{F}_{q^h}$-linear MDS codes, so $q^h+1$, such as Reed–Solomon codes~\cite{BLP}. This phenomenon arises from the gap between $n - \left\lceil \frac{r}{h} \right\rceil + 1$ and $n -  \frac{r}{h}  + 1$, and motivates the search for new constructions beyond traditional bounds.

In this work, we present explicit constructions of additive QMDS codes with length strictly greater than that of the longest known $\mathbb{F}_{q^h}$-linear MDS codes; see Section \ref{Section:constructions}. Such constructions necessitate a precise selection of subspaces exhibiting specific intersection characteristics. Their design illustrates that by adopting $\mathbb{F}_q$-linearity, it is possible to achieve extended code lengths while maintaining optimality relative to the generalized Singleton bound. In addition, we investigate the duals of these codes, providing conditions under which the dually QMDS property is satisfied. Our results are also connected with classical objects in finite geometry, such as dual hyperovals, dual arcs,  and spreads; see Section \ref{Section:Pre}.

Finally, it is worth mentioning that our results contribute to the growing body of research on additive and subfield-linear codes~\cite{martinezpenas2023unifying, martinezpenas2023weight}, and highlight the role of $\mathbb{F}_q$-linearity in overcoming classical constraints. This perspective opens new directions for constructing codes with strong distance properties and extended lengths, with potential applications in both classical and quantum communication.

It is worth noting that additive codes in $\mathbb{F}_{q^h}^n$ can also be viewed as linear codes endowed with the so-called \emph{folded Hamming distance}, as proposed in~\cite{martinezpenas2024linear}. This point of view unifies rank-metric and Hamming-metric theories under a common algebraic structure~\cite{martinezpenas2023unifying}. However, in this work we choose to focus on the $\mathbb{F}_q$-linear interpretation of these codes, emphasizing their subspace structure and geometric behavior over the base field.

\section{Preliminaries and Notation}\label{Section:Pre}

Let $q$ be a prime power, and let $\mathbb{F}_{q^h}$ denote the finite field extension of degree $h$ over $\mathbb{F}_q$. Let $\mathcal{B}=\{1,\xi,\xi^2,\dots,\xi^{h-1}\}$ be a basis of $\F_{q^h}$ over $\F_q$.

\subsection{Partial Spreads in Affine Geometry}

Let $\mathbb{F}_q^r$ denote the $r$-dimensional vector space over the finite field $\mathbb{F}_q$. A central concept in finite geometry is the notion of a partial spread: a collection of subspaces of fixed dimension that intersect trivially. These structures are especially relevant in coding theory, as they provide geometric configurations with strong independence properties.

\begin{definition}
A \emph{partial $t$-spread} in $\mathbb{F}_q^r$ is a collection $\mathcal{S}$ of $t$-dimensional $\mathbb{F}_q$-subspaces of $\mathbb{F}_q^r$ such that for all distinct $U, V \in \mathcal{S}$, we have $U \cap V = \{0\}$.
\end{definition}

A partial spread is said to be \emph{complete} or a \emph{full spread} if its elements partition the nonzero vectors of $\mathbb{F}_q^r$. This occurs precisely when $t$ divides $r$, in which case the full $t$-spread has size $(q^r - 1)/(q^t - 1)$. When $t$ does not divide $r$, complete spreads do not exist, and the study of partial spreads concerns determining the maximal number of pairwise disjoint $t$-subspaces.

A classical counting result provides an upper bound on the size of any partial $t$-spread.

\begin{theorem}
Let $1 \le t < r$ and suppose $t \nmid r$. Then any partial $t$-spread $\mathcal{S}$ in $\mathbb{F}_q^r$ satisfies
\[
|\mathcal{S}| \le \left\lfloor \frac{q^r - 1}{q^t - 1} \right\rfloor.
\]
\end{theorem}

The bound is tight when $t$ divides $r$, in which case a full $t$-spread exists and the floor is not necessary. In the general case, this upper bound is often far from {being} achievable, especially for small $q$ and large $r$.

{Also}, constructions of partial spreads have been found, so we have a lower bound for the size of a  partial $t$-spread with maximum cardinality. 

\begin{theorem}[\cite{Beutelspacher} Theorem 4.2]\label{beut}
Let $r = at + b$ with $0 < b < t$ and $a = \lfloor r/t \rfloor$. Then there exists a partial $t$-spread $\mathcal{S}$ in $\mathbb{F}_q^r$ such that
\[
|\mathcal{S}| = \sum_{i=1}^{a-1}q^{it+b}+1.
\]   
\end{theorem}

\subsection{Dimensional Dual Arcs and Dual Hyperovals}

\begin{definition}
Let \( V \) be a vector space of dimension \( r \) over the finite field \( \mathbb{F}_q \). A set \( {L} \) of \( (d+1) \)-dimensional subspaces of \( V \)  is called a {\( d \)-dimensional dual arc} (DDA) in \( V \) if it satisfies the following three conditions:

\begin{enumerate}
\item For any two distinct elements \( X, Y \in L \), the intersection \( X \cap Y \) is a 1-dimensional subspace.
\item For any three distinct elements \( X, Y, Z \in L \), the intersection \( X \cap Y \cap Z = \{0\} \).
\item All the elements of \( L \) generate  \( V \).
\end{enumerate}
\end{definition}

A dual arc \( L \) is said to be {complete} if it is not properly contained in any larger dual arc in the same ambient space.

\medskip

Let
\[
\theta_q(d) := \frac{q^{d+1} - 1}{q - 1}
\]
denote the number of 1-dimensional subspaces of a \( (d+1) \)-dimensional vector space over \( \mathbb{F}_q \). Then the following holds:

\begin{lemma}
If \( L \) is a \( d \)-dimensional dual arc in \( V \), then \( |L| \leq \theta_q(d) + 1 \).
\end{lemma}

\begin{definition}
A \( d \)-dimensional dual arc \( L \) is called a {dual hyperoval} (DHO) if \( |L| = \theta_q(d) + 1 \).
\end{definition}

\begin{definition}\cite{Dempwolff} A DHO $L$ in $V$ splits over the subspace  if there exists a subspace $Y$ such that $X\oplus Y=V$ for all $X\in L$. We call a DHO  
splitting if it splits over some subspace; otherwise we call it non-splitting.
\end{definition}

\begin{theorem}\cite[Proposition 2.9]{delfra}\label{noDHO} If there exists a $d$-dimensional DHO $L$ in $V\simeq\F_q^{2d+1}$, then $q$ is even.
\end{theorem}

\section{Additive codes}
For any additive code $C$ over a finite field $\mathbb{K}$, we can always find a field $\mathbb{F}_q$ that is the largest field between $\mathbb{F}_p$ and $\mathbb{K}$ such that $C$ is $\mathbb{F}_q$-linear. Write $\mathbb{K}$ as $\mathbb{F}_{q^h}$ for some positive integer $h$. \textbf{Throughout the rest part of this paper, without explicit statement, any additive code is an $\mathbb{F}_q$-linear subspace in $\mathbb{F}_{q^h}^n$ of dimension $r$ over $\mathbb{F}_q$ with minimum distance $d$.}
The triple $(n, r/h, d)$ denotes the length, normalized dimension, and minimum Hamming distance of such a code, and we refer to it as a code of type $[n, r/h, d]_q^h$. This notation highlights the ratio $r/h$, which may or may not be an integer. From now on $k$ will indicate $\lceil r/h\rceil$.

\begin{definition}
An additive code $C \subseteq \mathbb{F}_{q^h}^n$ is said to be \emph{integral} if $r/h \in \mathbb{N}$, and \emph{fractional} otherwise.
\end{definition}

The main parameter of interest is the minimum Hamming distance $d(C)$, defined in the usual way as
\[
d(C) := \min\{ \operatorname{wt}(c) : c \in C \setminus \{0\} \},
\]
where $\operatorname{wt}(c)$ counts the number of nonzero coordinates in $c$.

The dual of an additive $[n,r/h,d]^h_
 q$ code $C$ is defined by
 $$C^\perp=\{v\in\F_{q^h}^n\,:\, 
   \text{Tr}_{q^h/q}(\langle u,v\rangle) = 0, \textrm{ for all } u \in C\},$$
 where $\text{Tr}_{q^h/q}$ denotes the trace function from $\F_{q^h}$ to $\F_{q}$, and $\langle u,v\rangle$ denotes the classical inner product in $\F_{q^h}^n$. {Note that} $C^\perp$ is an additive $[n,n-r/h,d^\perp]^h_
 q$ code.
 
\begin{definition}
    Let $C_1,C_2$ be two additive codes in $\F_{q^h}^n$. We will say that $C_1$ and $C_2$ are equivalent if there exist $f_1,\dots,f_n$ $\F_q$-linear invertible maps and a permutaton $\sigma:[n]\rightarrow[n]$ such that $$C_1=\{\overline{f}(c^\sigma)\,:\,c\in C_2\},$$
    where $\overline{f}(c^\sigma):=(f_1(c_{\sigma(1)}),\dots,f_n(c_{\sigma(n)}))$.
\end{definition}
This definition is equivalent to the one in \cite[Definition 5]{martinezpenas2024linear}.

Since {any} $[n,r/h,d]_{q}^h$ additive code $C$ is an $\F_q$-subspace of $\F_{q^h}^n$,  it can be defined as the ${\mathbb F}_q$-subspace spanned by the rows of ${G}$, an $r \times n$ matrix with entries in ${\mathbb F}_{q^h}$, that we will call a generator matrix for $C$:

   $$G=\begin{pmatrix}
    \color{blue}a_{1,1}&\color{red}a_{1,2}&\cdots&\color{olive}a_{1,n}\\
    \color{blue}a_{2,1}&\color{red}a_{2,2}&\cdots&\color{olive}a_{2,n}\\
    \vdots&\vdots&\vdots&\vdots\\
    \color{blue}a_{r,1}&\color{red}a_{r,2}&\cdots&\color{olive}a_{r,n}\\
\end{pmatrix}\in M(r,n,\mathbb{F}_{q^h}).$$
The elements of ${G}$ are elements of ${\mathbb F}_{q^h}$, which we can split over the basis $\mathcal B$,
obtaining a $r \times nh$ matrix $\tilde G$ whose elements are from ${\mathbb F}_q$, and whose columns are grouped together in $n$ blocks of $h$ columns:

$$\tilde{G}=\begin{pmatrix}
    \color{blue}b^1_{1,1}&\color{blue}\cdots&\color{blue}b^h_{1,1}&\color{red}b^1_{1,2}&\color{red}\cdots&\color{red}b^h_{1,2}&\cdots&\color{olive}b^1_{1,n}&\color{olive}\cdots&\color{olive}b^h_{1,n}\\
    \color{blue}b^1_{2,1}&\color{blue}\cdots&\color{blue}b^h_{2,1}&\color{red}b^1_{2,2}&\color{red}\cdots&\color{red}b^h_{2,2}&\cdots&\color{olive}b^1_{2,n}&\color{olive}\cdots&\color{olive}b^h_{2,n}\\
    \color{blue}\vdots&\color{blue}\cdots&\color{blue}\vdots&\color{red}\vdots&\color{red}\cdots&\color{red}\vdots&\cdots&\color{olive}\vdots&\color{olive}\cdots&\color{olive}\vdots\\
    \color{blue}b^1_{r,1}&\color{blue}\cdots&\color{blue}b^h_{r,1}&\color{red}b^1_{r,2}&\color{red}\cdots&\color{red}b^h_{r,2}&\cdots&\color{olive}b^1_{r,n}&\color{olive}\cdots&\color{olive}b^h_{r,n}\\
\end{pmatrix}\in M(r,nh,\mathbb{F}_{q}).$$
$$\color{blue}{U_{1}(C,G)}\color{black}\hspace{1.6cm}\color{red}{U_2(C,G)}\color{black}\hspace{0.9cm}\dots\hspace{1.1cm}\color{olive}{U_n(C,G)}\color{black}\hspace{1.4cm}$$

Each $i$-th block of $h$ columns spans a subspace $U_i(C,G)$ (we can use $U_i(G),U_i(C)$ or $U_i$ when $C$ or $G$ are clear from the context) of $\F_q^r$ of dimension at most $h$. 

In \cite{BLP} the connection between $C$ and the multiset $\mathcal{X}_G(C):=\{U_1,\dots,U_n\}$ was studied, and it is noted that the choice of $G$ does not change the multiset $\mathcal{X}_G(C)$, so we can call it $\mathcal{X}(C)$. 

\begin{definition}[\cite{BLP}, Definition 4]
        An $h-(n,r,d)_q$ system is a multiset $S$ of $n$ subspaces of $\F_q^r$ of dimension at most $h$ such that each hyperplane of $\F_q^r$ contains at most $n-d$ elements of $S$ and some hyperplane contains exactly $n-d$ elements of $S$.
    \end{definition}

\begin{theorem}[\cite{BLP}, Theorem 5] \label{code-system}
If $C$ is an additive $[n,r/h,d]_{q}^h$ code, then ${\mathcal{X}}(C)$ is an $h-(n,r,d)_q$ system, and conversely, each  $h-(n,r,d)_q$ system defines an additive $[n,r/h,d]_{q}^h$ code.
\end{theorem}

    In this paper we will be more interested in the multiset of the subspaces $T_G(C):=\{W_i:=U_i^\perp\}_{i=1,\dots,n}$, where $U_i^\perp=\{v\in\F_q^r\,:\,\langle u,v\rangle=0\,\text{ for all }\,u\in U_i\}$ and $\langle\cdot,\cdot\rangle$ is the classic scalar product in $\F_q^r$. 

        \begin{definition}
        An additive code $C$ is said to be faithful if $\dim(W_i(C,G))=r-h$ for all $i=1,\dots,n$. It is said to be unfaithful otherwise.
    \end{definition}

\begin{definition}\cite{etzion}
A \emph{$t$-$(r,m,\lambda)_q$ subspace packing} is a collection $\mathcal{L}$ of $m$-dimensional subspaces of $\mathbb{F}_q^r$ (called \emph{blocks}) such that each $t$-dimensional subspace of $\mathbb{F}_q^r$ is contained in at most $\lambda$ blocks of $\mathcal{L}$. When $t,r,m$ are explicit by the context we will say that $\mathcal{L}$ is a $\lambda$-packing.
\end{definition}

Variants of this definition allow for blocks of varying dimension, for example $t$-$(r, \geq m, \lambda)_q$ packings, where each block has dimension at least $m$.

In our context, we are interested in the case $t=1$. Given an additive $[n, r/h, d]_{q^h}$ code $C$, the associated  multiset ${T}_G(C)$ of subspaces of $\mathbb{F}_q^r$ is a $1$-$(r, \geq r - h, n - d)_q$ subspace packing.

This perspective allows us to translate properties of additive codes (such as the minimum distance) into geometric constraints on the associated subspace configurations.

        \begin{definition}
            Given two $t$-$(r, \geq m, \lambda)_q$ packings, $T_1$ and $T_2$, they are equivalent if there exist an $\F_q$-linear bijection $\Phi:\F_q^r\longrightarrow\F_q^r$ such that $\Phi(T_1)=T_2$.
        \end{definition}
        As for $\mathcal{X}_G(C)$, the choice of the basis or the generator matrix does not actually change the multiset $T_G(C)$ (with a different generator matrix $G'$ we will have an equivalent $1$-$(r, \geq r-h, n-d)_q$ packing), so we will call it $T(C)$ when it is not important to specify the matrix $G$.

         Given a $1$-$(r, \geq r-h, n-d)_q$ packing $T$ it is easy to construct the corresponding $[n,r/h,d]_q^h$ additive code. We have to consider the set $\{W_i^\perp\,:\,W_i\in T\}$, and choose a set of $h$ generators for every $W_i^\perp$, and so construct the matrix $\tilde{G}$. Different choices of the generators will yield the construction of equivalent codes.
\begin{theorem}
    If $C$ is an additive $[n,r/h,d]_{q}^h$ code, then $T(C)$ is a $1$-$(r, \geq r-h, n-d)_q$ packing of size $n$, and conversely, each $1$-$(r, \geq r-h, n-d)_q$ packing of size $n$ defines an additive $[n,r/h,d]_{q}^h$ code up to equivalence.
\end{theorem}

We can also observe what happens with the equality.

\begin{prop}
    Let $C_1,C_2$ be two additive codes in $\F_{q^h}^n$ with respectively $G_1,G_2$ as generator matrices. Then $C_1$ and $C_2$ are equivalent if and only if there exist an $\F_q$-linear bijection  $\Phi:\F_{q}^r\rightarrow\F_{q}^r$ such that $\Phi(T_{G_1}(C_1))=T_{G_2}(C_2)$.
\end{prop}
\begin{proof}
   Consider $\overline{f}$ and $\sigma$ from the equivalence of $C_1$ and $C_2$. Let $u_1,\dots,u_r$ be the rows of $G_1$, then $\{\overline{f}u_1^\sigma,\dots,\overline{f}u_r^\sigma\}$  is a basis for $C_2$ and we will call $G_2'$ the matrix with these vectors as rows. There exists an $\F_q$-linear bijection  $\Phi:\F_{q}^r\rightarrow\F_{q}^r$ such that $\Phi(T_{G_2'}(C_2))=\Phi(T_{G_2}(C_2))$. Now we will show that $T_{G_2'}(C_2)=T_{G_1}(C_1)$. We want to split the elements of $G_1$ and $G_2'$ over the $\mathbb{F}_q$-basis $\mathcal{B}=\{1,\xi,\xi^2,\dots,\xi^{h-1}\}$ of $\mathbb{F}_{q^h}$. Let $u_{j,\ell}=(b_{j,\ell}^{1},\dots,b_{j,\ell}^h)$ for every $j=1,\dots,r$, $\ell=1,\dots,n$. For every $f_i$ there exists an associated invertible matrix $A_i\in GL(r,\mathbb{F}_q)$ such that $f_i(u_{j,\ell})=(b_{j,\ell}^{1},\dots,b_{j,\ell}^h)A_i$. Let $a_{\ell,\nu}:=(b_{1,\ell}^{\nu},\dots,b_{r,\ell}^\nu)$ for every $\ell=1,\dots,n$ and $\nu=1,\dots,h$. Recall that $V_{\ell,G_1}(C_1)$ is generated by the columns of the matrix $$(a_{\ell,1}\cdots a_{\ell,h}).$$ 
   
   Also, $V_{\sigma(\ell),G_2'}(C_2)$ is generated by $$(a_{\ell,1}\cdots a_{\ell,h})A_{\sigma(\ell)}$$ and since the matrix is invertible, this does not change the column space. So we have $V_{\sigma(\ell),G_2'}(C_2)=V_{\ell,G_1}(C_1)$, and the claim follows.

   The converse is trivial, recalling that different choices of the generators of the subspaces yield the construction of equivalent codes. 
   \end{proof}

Now we consider some properties of $T(C)$.

\begin{prop}\label{Prop c_i=0 then u in V_i^perp}
    Let C be an additive code. {Let  $c=uG\in C$, $u\in\F_{q}^r$, be a codeword.} Then $c_j=0$ if and only if $u\in W_j(C)$.
\end{prop}
\begin{proof}
    {Note that} $c_j=0$ if and only if $\sum_{i=1}^r u_ia_{i,j}=0$ that is equivalent to $\sum_{i=1}^r u_ib_{i,j}^e=0$ for any $e=1,\dots,h$. So $c_j=0$ if and only if $U_j(C)\subseteq u^\perp$ that is $u\in W_j(C)$. 
\end{proof}
        
\begin{cor}\label{cor Prop c_i=0 then u in V_i^perp}
    Let C be an additive code. Given $c=uG\in C$, with $u\in\F_{q}^r$, we have that $$\text{wt}(c)={n-}\#\{W\in{T}_G(C)\,:\,u\in W \}.$$
    \end{cor}
    
Additive codes satisfy a generalized Singleton bound, giving us the opportunity to define a type of codes that is in between MDS and almost MDS( or AMDS, which are linear codes with $d=n-k$). 

\begin{theorem}\cite[Theorem 10]{Huffman}
Let $C \subseteq \mathbb{F}_{q^h}^n$ be an additive code of dimension $r$ over $\mathbb{F}_q$. Then
\[
\left\lceil \frac{r}{h} \right\rceil \leq n - d(C) + 1.
\]
\end{theorem}

We will be particularly interested in codes satisfying this relaxed form of the Singleton bound.

\begin{definition}
An additive code $C$ is called \emph{quasi-MDS (QMDS)} if
\[
d(C) = n - \left\lceil \frac{r}{h} \right\rceil + 1.
\]
    An $[n,r/h,d]_{q}^h$ additive QMDS code $C$ is called ``long'' if $n>q^h+1$.   An additive code C is said to be dually QMDS if both $C$ and $C^\perp$ are QMDS.
\end{definition}
 
\section{Constructions}\label{Section:constructions}
Our goal for this section is to construct $[n,r/h,d]_q^h$ QMDS additive codes, with $n>q^h+1$. In \cite{BLP} an upper bound on $n$ was found.

\begin{theorem}\cite{BLP}\label{boundBLP} If there is an $[n,r/h,d]_q^h$ QMDS additive code then
$$n\leq k-2+q^h+\frac{q^h-1}{q^{r_0}-1},$$
where $r=(k-1)h+r_0$ with $1\leq r_0\leq h$.
    
\end{theorem}

\begin{cor}
    If there is an $[n,r/h,d]_q^h$ dually QMDS additive fractional code $C$ then 
    $$k\leq q^h+\frac{q^h-1}{q^{r_0}-1}-1.\label{boundDim}$$
\end{cor}
\begin{proof}
    The proof is trivial applying Theorem \ref{boundBLP} to $C^\perp$.
\end{proof}

As we said we will construct QMDS additive codes with $n>q^h+1$, in some specific cases we will construct codes with $n=q^h+f(q)$ where $f$ is an increasing function.
To construct these codes we will actually construct $t$-packings of subspaces of $\F_q^r$.
\subsection{Construction of ``long'' QMDS additive codes with  $q\geq2,h\geq2,q^h-1\geq k\geq 2$, and $n=q^h+2$}

Let ${r_0}\leq h/2$ and $r:=(k-1)h+{r_0}$. Let $S_1,S_2$ be two $\F_q$-subspaces of $\F_{q^h}$ of $q$-dimension ${r_0}$, such that $S_1\cap S_2=\{0\}$. 

    Let $L:=\{W_\alpha\,:\,\alpha\in\F_{q^h}\}\cup\{W_{\infty_1},W_{{\infty_2}}\}$, where 
        $$W_\alpha:=\{(x_1,\dots,x_{k-1},z_1,\dots,z_{r_0})\,:\sum_{i=1}^{k-1}\alpha^{i-1}x_i+\alpha^{k-1}\sum_{j=1}^{r_0}\xi^{j-1}z_j=0\}\subseteq(\F_{q^h})^{k-1}\times(\F_q)^{r_0},$$
    $$W_{\infty_1}:=\{(x_1,x_2,\dots,x_{k-2},s,0,\dots,0)\,:\,s\in S_1\}\subseteq(\F_{q^h})^{k-1}\times(\F_q)^{r_0},$$
    $$W_{\infty_2}:=\{( x_1,x_2,\dots,x_{k-2}, s,0,\dots,0)\,:\,s\in S_2\}\subseteq(\F_{q^h})^{k-1}\times(\F_q)^{r_0}.$$
        
        Note that $W_\alpha,W_{{\infty_1}},W_{{\infty_2}}$ are $\F_q$-subspaces of $(\F_{q^h})^{k-1}\times(\F_q)^{r_0}$, with $\dim_{\F_q}(W_\alpha)=\dim_{\F_q}(W_{\infty_i})=(k-2)h+{r_0}$ and so they are isomorphic to some linear subspaces of $\F_q^r$ of dimension $(k-2)h+{r_0}$.

\begin{theorem}\label{Thm_n=q^h+2}
    The set $L$ as defined above is a $(k-1)$-packing set.
\end{theorem}
\begin{proof}
We want to prove that the intersection of $k$ of these subspaces is the {zero} vector. We divide the proof in cases. In the proof we will call $x_k:=\sum_{j=1}^{r_0}\xi^{j-1}z_j\in\F_{q^h}$.

\begin{itemize}
    \item $K:=W_{\alpha_1}\cap\dots\cap W_{\alpha_k}$:

    Given $v:=(x_1,\dots,x_{k-1},z_1,\dots,z_{r_0})\in K$, we obtain a linear system in $x_1,\dots,x_k$ with as the associated matrix, the Vandermonde matrix $VM(\alpha_1,\dots,\alpha_k)$, so the unique solution is $x_1=x_2=\cdots=x_k=0$, and since $x_k=0$ implies $z_1=\cdots=z_{r_0}=0$, we obtain $v=\mathbf{0}$.

\item $K:=W_{\infty_1}\cap W_{\alpha_2}\cap\dots\cap W_{\alpha_k}$:

    Given $v:=(x_1,\dots,x_{k-1},z_1,\dots,z_{r_0})\in K$, firstly since $v\in W_{\infty_1}$ we obtain $x_k=0$. Then we obtain a linear system in $x_1,\dots,x_{k-1}$ with as the associated matrix, the Vandermonde matrix $VM(\alpha_2,\dots,\alpha_k)$, so the unique solution is $x_1=x_2=\cdots=x_{k-1}=0$, and so $v=\mathbf{0}$.

\item $K:=W_{\infty_2
}\cap W_{\alpha_2}\cap\dots\cap W_{\alpha_k}$:

Its proof is analogous to the previous case.

\item $K:=W_{\infty_1}\cap W_{\infty_2}\cap  W_{\alpha_3}\cap\dots\cap W_{\alpha_k}$:

  Given $v:=(x_1,\dots,x_{k-1},z_1,\dots,z_{r_0})\in K$, firstly since $v\in W_{\infty_1}$ we obtain $x_k=0$. Note also that there exist $s_1\in S_1,s_2\in S_2$ such that $s_1=x_{k-1}= s_2$, and so $x_{k-1}=0$. Lastly we obtain a linear system in $x_1,\dots,x_{k-2}$ with as the associated matrix, the Vandermonde matrix $VM(\alpha_3,\dots,\alpha_k)$, so the unique solution is $x_1=\cdots=x_{k-2}=0$, and so $v=\mathbf{0}$.\qedhere
\end{itemize}
\end{proof}

\begin{remark}
    Since $L$ is a $(k-1)$-packing set, it is a $1$-$(r, \geq r-h, k-1)_q$ packing of size $q^h+2$. This is equivalent to have an $[n,r/h,d]_q^h$ additive code, where $n=q^h+2$ and $n-d=k-1$. {Thus} it is a QMDS additive code {longer than the} Reed-Solomon Code. We will call $\mathcal{A}_{q,h,k,r_0}$, or simply $\mathcal{A}_{q,h,k}$ if $r_0$ it is not important in the context,  the code associated to $L$.
\end{remark}

\begin{prop}
    If $k=3$ then $L$ is not extendable.
\end{prop}
\begin{proof}
    Consider  an $\F_q$-subspace {$W$} of $\F_{q^h}\times\F_{q^h}\times(\F_{q})^{r_0}$ of $q$-dimension $h+{r_0}$ different from all the subspaces of $L$. By contradiction, let $N:=L\cup\{W\}$ be a $2$-packing set. So we have that $W\cap W_{\infty_1}\cap W_{\infty_2}=\{\mathbf{0}\}$. {Since}  $W_{\infty_1}\cap W_{\infty_2}=\{(x,0,\mathbf{0})\, :\,x\in\F_{q^h}\}$,  $W=\{(f(x,\overline{z}),x,\overline{z}):x \in\F_{q^h},\overline{z}\in(\F_q)^{r_0}\}$ {for some function $f: \mathbb{F}_{q^h}\times \mathbb{F}_q^{r_0} \to \mathbb{F}_{q^h}$}. Fix $x_0\in S_1^*$. Consider $v_0:=(f(x_0,\mathbf{0}),x_0,\mathbf{0})$ and $\alpha_0:=-f(x_0,\mathbf{0})/x_0$. So we have a contradiction since $v_0\in W\cap W_{\infty_1}\cap W_{\alpha_0}$.   
\end{proof}

\subsection{Construction of ``long'' QMDS additive codes with any $q\geq2,q^h-1\geq k\geq3,h\geq6,h\neq7$, and $n=q^h+3$}

Let ${r_0}\leq h/6$ and $r:=(k-1)h+{r_0}$. Let $r_1,r_2$ be such that $r_1\leq2h/3$, $r_2\leq h/2$, and $r_1+r_2=h+{r_0}$ (if $6\mid h$ a good choice is ${r_0}=h/6$, $r_1=2h/3$, $r_2=h/2$). Let $Y_1,Y_2,Y_3$ be three $\F_q$-subspaces of $\F_{q^h}$ of $q$-dimension $r_1$ such that $Y_1\cap Y_2 \cap Y_3=\{0\}$, $S_1,S_2,S_3$ be three $\F_q$-subspaces of $\F_{q^h}$ of $q$-dimension $r_2$ such that $S_1\cap S_2=S_1 \cap S_3=S_2\cap S_3=\{0\}$.

    Let $L:=\{W_\alpha\,:\,\alpha\in\F_{q^h}\}\cup\{W_{\infty_1},W_{{\infty_2}},W_{{\infty_3}}\}$, where 
        $$W_\alpha:=\{(x_1,\dots,x_{k-1},z_1,\dots,z_{r_0})\,:\sum_{i=1}^{k-1}\alpha^{i-1}x_i+\alpha^{k-1}\sum_{j=1}^{r_0}\xi^{j-1}z_j=0\}\subseteq(\F_{q^h})^{k-1}\times(\F_q)^{r_0},$$
    $$W_{\infty_1}:=\{(x_1,x_2,\dots,x_{k-3},y,s,0,\dots,0)\,:\,y\in Y_1,\,s\in S_1\}\subseteq(\F_{q^h})^{k-1}\times(\F_q)^{r_0},$$
$$W_{\infty_2}:=\{(x_1,x_2,\dots,x_{k-3},y,s,0,\dots,0)\,:\,y\in Y_2,\,s\in S_2\}\subseteq(\F_{q^h})^{k-1}\times(\F_q)^{r_0},$$
$$W_{\infty_3}:=\{(x_1,x_2,\dots,x_{k-3},y,s,0,\dots,0)\,:\,y\in Y_3,\,s\in S_3\}\subseteq(\F_{q^h})^{k-1}\times(\F_q)^{r_0}.$$
        
        Note that $W_\alpha,W_{{\infty_1}},W_{{\infty_2}},W_{{\infty_3}}$ are $\F_q$-subspaces of $(\F_{q^h})^{k-1}\times(\F_q)^{r_0}$, with $\dim_{\F_q}(W_\alpha)=(k-2)h+{r_0}$ and $\dim_{\F_q}(W_{\infty_i})=(k-3)h+r_1+r_2=(k-2)h+{r_0}$. Thus they are isomorphic to some linear subspaces of $\F_q^r$ of dimension $(k-2)h+{r_0}$.

\begin{theorem}\label{Thm_n=q^h+3}
    {The set $L$ defined above} is a $(k-1)$-packing set.
\end{theorem}
\begin{proof}
We want to prove that {any $k$ of these subspaces intersect trivially}. We divide the proof in {a number of } cases. In the proof we will call $x_k:=\sum_{j=1}^{r_0}\xi^{j-1}z_j$.

\begin{itemize}
    \item $K:=W_{\alpha_1}\cap\dots\cap W_{\alpha_k}$:

    The proof is the same as the one in Theorem \ref{Thm_n=q^h+2}.

\item $K:=W_{\infty_i}\cap W_{\alpha_2}\cap\dots\cap W_{\alpha_k}$:

    The proof is the same as the one in Theorem \ref{Thm_n=q^h+2}.

\item $K:=W_{\infty_i}\cap W_{\infty_j}\cap  W_{\alpha_3}\cap\dots\cap W_{\alpha_k}$:

 The proof is the same as the one in Theorem \ref{Thm_n=q^h+2}.

  \item $K:=W_{\infty_1}\cap W_{\infty_2}\cap W_{\infty_3}\cap W_{\alpha_4}\cap\dots\cap W_{\alpha_k}$:

  Given $v:=(x_1,\dots,x_{k-1},z_1,\dots,z_{r_0})\in K$, since $v\in W_{\infty_1}\cap W_{\infty_2}$ we obtain $x_k=x_{k-1}=0$. {Also,} $x_{k-2}\in Y_1\cap Y_2\cap Y_3=\{0\}$. {Finally,} we obtain a linear system in $x_1,\dots,x_{k-3}$ with as the associated matrix, the Vandermonde matrix $VM(\alpha_4,\dots,\alpha_k)$, so the unique solution is $x_1=\cdots=x_{k-3}=0$, and {and thus} $v=\mathbf{0}$. \qedhere
\end{itemize}
\end{proof}

\begin{remark}
Since $L$ is a $(k-1)$-packing set, it is a $1$-$(r, \geq r-h, k-1)_q$ packing of size $q^h+3$. This is equivalent to have an $[n,r/h,d]_q^h$ additive code, where $n=q^h+3$ and $n-d=k-1$. We will call $\mathcal{B}_{q,h,k,r_0}$, or simply $\mathcal{B}_{q,h,k}$ if $r_0$ is not important in the context,  the code associated to $L$.
\end{remark}

\begin{prop}
    If $k=4$ then $L$ is not extendable.
\end{prop}
\begin{proof}
    Consider $W$, an $\F_q$-subspace of $(\F_{q^h})^3\times(\F_{q})^{r_0}$ of $q$-dimension $2h+{r_0}$ {not in} $L$. By contradiction, let $N:=L\cup\{W\}$ be a $3$-packing set. {Therefore} $W\cap W_{\infty_1}\cap W_{\infty_2}\cap W_{\infty_3}=\{\mathbf{0}\}$ and  $W_{\infty_1}\cap W_{\infty_2}\cap W_{\infty_3}=\{(x,0,0,\mathbf{0})\, :\,x\in\F_{q^h}\}$. {This} means that $W=\{(f(x,y,\overline{z}),x,y,\overline{z}):x,y \in\F_{q^h},\overline{z}\in(\F_q)^{r_0}\}$ {for some function $f: \mathbb{F}_{q^h}\times \mathbb{F}_{q^h}\times\mathbb{F}_q^{r_0}\to \mathbb{F}_{q^h}$}. Fix $x_0\in Y_1^*\cap Y_2^*$. Consider{ing} $v_0:=(f(x_0,\mathbf{0}),x_0,0,\mathbf{0})$ and $\alpha_0:=-f(x_0,\mathbf{0})/x_0$, we have a contradiction since $v_0\in W\cap W_{\infty_1}\cap W_{\infty_2}\cap W_{\alpha_0}$.
    
\end{proof}

\subsubsection{Construction of ``long'' QMDS additive codes with $k=3$, any $q\geq2,h\geq6,h\neq7$ and $n=q^h+g(q)$}
In the case of $k=3$ we can add more subspaces of the type $W_{\infty_i}$. 
Consider the largest set $\Omega_1=\{
Y_i\}_{i\in I_1}$ consisting of $\F_q$-subspaces of $\F_{q^h}$ of $q$-dimension $r_1$ such that $Y_{i_1}\cap Y_{i_2} \cap Y_{i_3}=\{\mathbf{0}\}$ for any {pairwise distinct} ${i_1},{i_2},{i_3}\in I_1$. Also, consider the largest set $\Omega_2=\{S_i\}_{i\in I_2}$ consisting of $\F_q$-subspaces of $\F_{q^h}$ of $q$-dimension $r_2$ such that $S_i\cap S_j=\{\mathbf{0}\}$ for any $i\neq j\in I_2$.

{Clearly,} $\Omega_2$ is partial $r_2$-spread of $\F_{q^h}$. Let us consider a maximal partial $(h-r_1)$-spread of $\F_{q^h}$, $\Gamma=\{U_j\}_{j\in J}$. Recall that $r_1\leq2h/3,r_2\leq h/2,r_1+r_2=h+r_0$. Since $r_2\geq h-r_1$, we have $\#\Omega_2\leq\Gamma$. Now consider $\Gamma^\perp:=\{U_j^\perp\}_{j\in J}$: it is a set of $r_1$-subspaces such that the intersection of three of them is the trivial subspace. So we have that $\#\Omega_1\geq\#\Gamma^\perp=\#\Gamma\geq\#\Omega_2$. 

{Let}  $g(q):=\#\Omega_2$ {and consider}
$$W_{\infty_i}:=\{(y,s,0,\dots,0)\,:\,y\in Y_i,s \in 
S_i\}\subseteq\F_{q^h}^2\times(\F_q)^{r_0},$$
and  $\overline{L}:=\{W_\alpha\,:\,\alpha\in\F_{q^h}\}\cup\{W_{\infty_i}\,:\,i=1,\dots,g(q)\}$.

\begin{theorem}
    {The set $\overline{L}$ defined above} is a $2$-packing set.
\end{theorem}
\begin{proof}
    We have to show that any three subspaces have trivial intersection.
    If at least one of these three subspaces are of the type $W_\alpha$ then the proof follows the one of Theorem \ref{Thm_n=q^h+3}. {The} case $W_{\infty_{i_1}}\cap W_{\infty_{i_2}}\cap W_{\infty_{i_3}}$ {follows from} the assumptions on $Y_{i_1},Y_{i_2},Y_{i_3}$ and $S_{i_1},S_{i_2}$.
\end{proof}

As before, by Theorem \ref{beut}, given $a,b$ non negative integers such
that $h=ar_2+b$, we obtain $$g(q)\geq \sum_{i=1}^{a-1}q^{ir_2+b}+1.$$ 

\begin{remark}
Since $L$ is a $2$-packing set,  it is a $1$-$(r, \geq r-h, 2)_q$ packing of size $q^h+g(q)$. This is equivalent to have an $[n,2+r_0/h,d]_q^h$ additive code, where $n=q^h+g(q)$ and $n-d=2$. This means {that} it is a QMDS additive code that we will indicate with $\overline{\mathcal{B}}_{q,h,r_0}$, or simply $\overline{\mathcal{B}}_{q,h}$ if $r_0$ is not important in the context.
\end{remark}

\subsection{Construction of ``long'' QMDS additive codes with $k=2$,  $q\geq2,h\geq2$, and $n=q^h+f(q)$}
In the case of $k=2$, we are looking for some $1$-$(r, \geq r-h, 1)_q$ packings. We can notice that a $1$-packing it is actually a partial spread of $\F_q^r$. So given $\Omega$, a $r_0$-dimensional partial spread of $\F_q^r$, where $r=h+r_0$, of cardinalty $n$, we can construct an $[n,1+r_0/h,d]_q^h$ QMDS additive code, that we will call $C_\Omega$. Consider the $r_0$-dimensional partial spread of maximum cardinality $\tilde{\Omega}$, this gives us a $[\#\tilde{\Omega},1+r_0/h,d]_q^h$ QMDS additive code. Let $\#\tilde{\Omega}=q^h+f(q)$, we can obtain a lower bound for $f(q)$. Let $a,b$ be non negative integers such that $h=ar_0+b$ with $b\leq r_0-1$. From Theorem \ref{beut} we obtain  $$f(q)\geq \sum_{i=1}^{a-1}q^{ir_0+b}+1.$$ 

Obviously, in the case $r_0\mid h$, the Desarguesian spread is the optimal choice for $\tilde{\Omega}$, {and} in this case we have
$$f(q)=\sum_{i=1}^{a-1}q^{ir_0}+1=\frac{q^h-1}{q^{r_0}-1},$$
and we obtain a code we the same parameters as the one constructed in \cite[Theorem 23]{BLP}.

\section{Dual codes}

The aim of this section is to provide examples of QMDS codes that are not dually QMDS.

\begin{definition}\cite[Definition 20]{BLP} 
    Let $J$ be a subset of $\{1,\dots,n\}$. The geometric quotient at $J$ of an additive $[n,r/h,d]_q^h$ code C is a code $C/J$ of length $n-| J|$ defined as
    $$C/J=\{c_J\,:\,c\in C, c_j=0\,\text{ for all }\,j\in J\},$$
    where $c_J$ denotes the vector obtained from $c$ by deleting the coordinate positions belonging
to the subset $J$. We say a geometric quotient is {non-obliterating} if $|C/J|\geq q^{h}$.
\end{definition}
\begin{remark}
    Note that an obliterating quotient always yields an unfaithful additive code, since if we project the code in one of its components we will obtain $\dim U_j<h$.
\end{remark}

Our goal is to see what happens to the associated dual projective system $T_G(C)$ when we apply this operation to $C$. We can consider the geometric quotient with $|J|=1$, and so w.l.o.g. $J=\{1\}$, and define the operation inductively.  

\begin{prop}\label{propTcGIsom}
    Let $C$ be an $[n,r/h,d]_q^h$ additive code. Then $C/\{1\}$ is an $[n-1,\ell/h,d']_q^h$ additive code, where $\ell=\dim(W_1(C))$, and $W_j(C/\{1\})\simeq W_1(C)\cap(W_{j+1}{(C)})$ for any $j=1,\dots,n-1$.
\end{prop}
\begin{proof}
Consider the matrix $G_{\{1\}}$, obtained by deleting the first column from $G$.
From Proposition $\ref{Prop c_i=0 then u in V_i^perp}$ we have that 
$$C/\{1\}=\{c_{\{1\}}\,:\,c\in C\,,\, c_1=0\}=\{uG_{\{1\}}:u\in U_1(C)^\perp=W_1(C)\},$$
so since $G$ is full rank, $\dim(C/\{1\})=\dim (W_1(C))$.
Let $\{u_1,\dots,u_{\ell}\}$ be a basis of $W_1(C)$. Consider the matrix $M$ that has as rows the vectors $\{u_i\}_{i=1,\dots,\ell}$, so $\Phi(v):=vM$ will be an isomorphism between $\F_q^{\ell}$ and $W_1(C)$. 

Let $N$ be the matrix product of $M$ and $G_{\{1\}}$.
We have
$$C/\{1\}=\{vN:v\in \F_q^\ell\}.$$
 {Thus} $N$ is a generator matrix for $C/\{1\}$. Given this, we know that $U_j(C/{\{1\}})=\{Mv\,:\,v \in U_{j+1}(C)\}$ for any $j=1,\dots,n-1$ {and} we will prove that $\Phi(W_{j}(C/{\{1\}}))=W_1(C)\cap W_{j+1}(C)$. Let $v_0\in\Phi(W_{j}(C/{\{1\}}))$. By the definition of $\Phi$, we have $v_0\in W_1(C)$ and  there exists $w\in W_{j}(C/\{1\})$ such that $v_0=wM$. We know that $\langle w,u\rangle =0$ for any $u\in U_{j}(C/\{1\})$. So $w(Mv)=(wM)v=0$ for any $v \in U_{j+1}(C)$ and so $v_0=wM\in W_{j+1}(C)$. We can prove the converse in a similar way.
 \end{proof}

\begin{remark}\label{Remark dual}
     {Note that} $C/\{1\}$ is faithful if and only if  $\dim(W_1(C)\cap W_j(C))=\dim(W_1(C))-h$ for any ${j=2,\dots,n}$. In general, given $J=\{j_1,\dots,j_l\}$, $C/{J}$ is faithful if and only if $$\dim(W_{j_1}(C)\cap\dots\cap W_{j_l}(C)\cap W_j(C))=\\dim(W_{j_1}(C)\cap\dots\cap W_{j_l}(C))-h$$ for any $j\in[n]\setminus J$.
\end{remark}
From now on we will denote $C/{\{j\}}$ also by $C/W_j$ and  we will use these results about faithful codes.
\begin{theorem}  \label{thm BLP c faith}
Let $C$ be an additive $[n,r/h,d]_{q}^h$ code, and let $d^\perp$ denote the minimal distance of $C^\perp$.
\begin{enumerate}
    \item   $C$ is unfaithful if and only if $d^\perp=1$; \cite[Theorem 17]{BLP}.
    \item   $C$ is faithful with $d \geq 2$ if and only if $C^{\perp}$ is faithful with $d^{\perp} \geq 2$; \cite[Theorem 18]{BLP}.
\end{enumerate}
\end{theorem}

\begin{theorem}\cite[Theorem 22]{BLP} \label{Dual faith mds}
    Let $C$ be an additive QMDS code.
The dual code $C^{\perp}$ is an additive faithful QMDS code if and only if every non-obliterating geometric quotient of $C$ is faithful.
\end{theorem}

{We are in position now to prove the main results of this section.}

The following theorem {shows the existence of} examples of QMDS codes that are not dually QMDS {and} this can happen only for fractional additive QMDS codes {\cite[Theorem 3.3]{YADAV}}.
\begin{theorem}\label{DualNONmds}
    {The codes $\mathcal{A}_{q,h,k}^\perp$, $\mathcal{B}_{q,h,k}^\perp$, and $\overline{\mathcal{B}}_{q,h}^\perp$ are not QMDS.}
\end{theorem}
\begin{proof}
From Theorem \ref{thm BLP c faith} we have that $\mathcal{A}_{q,h,k}^\perp,\mathcal{B}_{q,h,k}^\perp$ are faithful.
    We will prove now that for both these codes $C/W_{\infty_1}$ is unfaithful.  
    
    For $C=\mathcal{A}_{q,h,k}$ we have $\dim_{\F_q}(W_{\infty_1}\cap W_{\infty_2})=(k-2)h$ and $\dim_{\F_q}(W_{\infty_1}\cap W_{\alpha=0})=(k-3)h+ r_0$, so from the osservation made before we have that $C/W_{\infty_1}$ is unfaithful.

    Now let $C=\mathcal{B}_{q,h,k}$. Note that from the condition on $r_0,r_1,r_2$ we obtain $r_1=h+r_0-r_2\geq h/2+r_0$, so $\dim_{\F_q}(W_{\infty_1}\cap W_{\infty_2})\geq(k-3)h+2r_0$. {Also,} we have  that $\dim_{\F_q}(W_{\infty_1}\cap W_{\alpha=0})=(k-4)h+r_1+r_2=(k-3)h+r_0$, and the claim follows as in the previous case. 
    
    Finally, let $C=\overline{\mathcal{B}}_{q,h}^\perp$. Now, $C/W_{\alpha=0}$ is unfaithful since $\dim(W_{\alpha=0}\cap W_{\alpha=1})=r_0$ and $\dim(W_{\alpha=0}\cap W_{\infty_1})=r_2\geq h/3+r_0>r_0$.
\end{proof}



Thanks to the following theorem, we can exhibit a family of dually QMDS codes.
\begin{theorem}
    Let $C$ be an $[n,r/h,d]$ faithful additive QMDS code with $r=h+r_0$, $1\leq r_0<h$. Then $C^\perp$ is an $[n,n-r/h,d^\perp]$ additive QMDS code.
\end{theorem}
\begin{proof}
    Since $C$ is faithful, then $\dim(W_j(C))=r-h=r_0<h$ for every $j=1,\dots,n$. {Since} every geometric quotient of $C$ is obliterating, {the claim follows by} Theorem \ref{Dual faith mds}.
\end{proof}

\begin{cor} Given $\Omega$, an $r_0$-dimensional partial spread of $\F_q^r$, 
     {the code} $C_\Omega^\perp$ is a $[q^h+f(q),n-1-r_0/h,2]_q^h$ QMDS code.
\end{cor}


In the following, we want to study the packing $T(C)$ associated to a dually QMDS code. Since we are interested in non-trivial codes, we are interested in codes $C$ such that $d,d^\perp>1$, so by Theorem \ref{thm BLP c faith} we are interested in faithful codes. We propose here another proof of  \cite[Theorem 1]{martinezpenas2024linear} using our notation to help the reader understand better the connections in Section \ref{sectiondda}.
\begin{theorem}\label{iffDQMDS}
    Given  an $[n,r/h,d]$ faithful QMDS code $C$ with $d>1$,  the following conditions are  equivalent:
    \begin{itemize}
        \item[a)] $C$ is dually QMDS;
        \item[b)] $\dim(\bigcap_{j\in J}W_j)=r-|J|h$ for any $J\in P_{\leq k-1}([n])$ where $ P_{\leq k-1}([n])$ indicates the set of subsets of $[n]$ of cardinality at most $k-1$.
    \end{itemize}
\end{theorem}
\begin{proof}

    \item $(a)\Rightarrow(b)$. The proof will be by induction on the cardinality of $J$.

    If $|J|=1$ then it is trivial since $C$ is faithful. Let $|J|=\ell>1$, and w.l.o.g. consider $J=[\ell]$. By contradiction let $\dim(\bigcap_{j=1,\dots,\ell}W_j)>r-\ell h$. We will prove that $C/[\ell-1]$ is unfaithful, and so by Theorem \ref{Dual faith mds} we will obtain a contradiction. We know that $\dim(C/[\ell-1])=\dim(\bigcap_{j=1,\dots,\ell-1}W_j)=r-(\ell-1)h$, so to be faithful we need that $\dim(W)=r-(\ell-1)h-h$ for any $W\in T(C/[\ell-1])$. By Remark \ref{Remark dual} we know that there exist $\overline{W}\in T(C/[\ell-1])$ such that $\overline{W}\simeq(\bigcap_{j=1,\dots,\ell-1}W_j)\cap W_\ell$, and so $\dim(\overline{W})>r-(\ell-1)h-h$ giving us the unfaithfulness of $C/{[\ell-1]}$, a contradiction.
    
    $(b)\Rightarrow(a)$. Suppose that $C^{\bot}$ is not QMDS. By Theorem \ref{Dual faith mds} there is a non-obliterating geometric quotient $C/J$, with $J=\{j_1,\dots,j_{\ell-1}\}$ that is unfaithful. Then, by  Remark \ref{Remark dual}, there exists $j_\ell\in[n]$ such that $\dim(\bigcap_{j=j_1,\dots,j_\ell}W_j)>r-\ell h$, a contradiction.
\end{proof}

\begin{cor}\label{geomQuotQMDS}
    The geometric quotient $C/J$ of a dually QMDS code $C$ is dually QMDS.
\end{cor}
\begin{proof}
    We can prove it for $|J|=1$ since we can define the geometric quotient inductively. Without loss of generality, let $J=\{1\}$.

    From Proposition \ref{propTcGIsom} we know that $W_j(C/J)\simeq W_1(C)\cap(W_{j+1}(C))$. So $$\dim\Big(\bigcap_{j\in J_0}W_j(C/J)\Big)=\dim\Big(\bigcap_{j\in J_0}W_j(C)\cap W_1(C)\Big)=\dim\Big(\bigcap_{j\in J_0\cup\{1\}}W_j(C)\Big),$$ for any $J_0\subseteq\{2,\dots,n\}$. From  Theorem \ref{iffDQMDS} we have $$\dim\Big(\bigcap_{j\in J_0\cup\{1\}}W_j(C)\Big)=r-(|J_0|+1)h=(r-h)-|J_0|h=\dim(C/J)-|J_0|h,$$ for any $J_0\subseteq\{2,\dots,n\}$. By Theorem \ref{iffDQMDS},  we obtain that $C/J$ is dually QMDS.
\end{proof}

\section{Connection between additive codes and DDA}\label{sectiondda}
As a byproduct of Theorem \ref{iffDQMDS} we have:
\begin{cor}\label{DDAtoCode}
    An $h$-dimensional dual arc $L$ in $\F_q^{2h+1}$ of cardinality $n$ is equivalent to an $[n,2+1/h,d]_q^h$ dually QMDS code. 
\end{cor}
\begin{proof}
    It is sufficient to observe that $L$ is a $2$-packing that satisfies condition  $(b)$ in Theorem \ref{iffDQMDS}.
\end{proof}

We focus now on $[n,2+1/h,d]_q^h$ dually QMDS codes. In this case, by   \cite[Theorem 14]{BLP} we obtain $n\leq q^h+\frac{q^h-1}{q-1}+1=\theta_q(h)+1$. So the existence of a $[{\theta_q(h)+1},2+1/h,d]_q^h$ dually QMDS code is equivalent to the existence of a DHO. 

A trivial corollary of Theorem \ref{noDHO} is:
\begin{cor}\label{noDQMDS3}
    There is no $[\theta_q(h)+1,2+1/h,d]_q^h$ dually QMDS code when $q$ is odd.
\end{cor}

Now let us consider the more general case $r=(k-1)h+1$ with $k\geq3$.  By   \cite[Theorem 14]{BLP}, for an $[n,(k-1)+1/h,d]_q^h$ dually QMDS code, $n\leq q^h+\frac{q^h-1}{q-1}+k-2=\theta_q(h)+k-2$.

\begin{cor}
    There is no $[\theta_q(h)+k-2,(k-1)+1/h,d]_q^h$ dually QMDS code when $q$ is odd.
\end{cor}
\begin{proof}
    We will use induction on $k$. Corollary \ref{noDQMDS3} is the base case. 

    By contradiction, let $C$ be a $[\theta_q(h)+k-2,(k-1)+1/h,d]_q^h$ dually QMDS code.  By Corollary \ref{geomQuotQMDS} we have that $C/\{1\}$ is a $[\theta_q(h)+k-3,(k-2)+1/h,d]_q^h$ dually QMDS code, a contradiction to the induction hypothesis.
\end{proof}

\subsection{Constructions with $h$-dimensional DHO} In  \cite[Example 2.5]{CoopersteinThas} an $h$-dimensional DHO in $\F_2^{2h+1}$ is constructed. From Corollary \ref{DDAtoCode} this corresponds to a $[2^{h+1},2+1/h,d]_2^h$ dually QMDS code. In \cite{BLP} a QMDS code with the same parameter was found. If we consider the dual code in the case $h=2$, we have an $[8,5.5,3]_2^2$ QMDS code, which reach the bound in Corollary \ref{boundDim} and that is also a counterexample for a bound on $k$ found in \cite[Theorem 16]{BLP} (the authors of the paper have published an erratum where they restate the theorem and it holds only in the integral case \cite{BLP2}).

\section*{Open Problems}

Our results demonstrate the existence of additive QMDS codes with length strictly greater than \(q^h + 1\), and establish connections between such codes and geometric objects such as subspace packings. Nonetheless, several natural problems remain open:

\begin{enumerate}
    \item \textbf{Upper bounds on \(k\) for fractional QMDS codes.} 
    For additive QMDS codes of type \([n, r/h, d]_{q^h}\) with fractional rate \(r/h \notin \mathbb{N}\), no general upper bound on \(k = \lceil r/h \rceil\) is known. 

    \item \textbf{Constructions with \(n = q^h + f(q)\) for \(k > 3\).}
    Existing constructions of ``long'' additive QMDS codes with \(n = q^h + f(q)\) are limited to the case \(k \leq 3\). For \(k > 3\), there are no known example. Does them exist? If yes, can one devise general constructions yielding \(n = q^h + f(q)\) with \(f(q) \to \infty\) and \(k \geq 4\),?

   \item \textbf{Structure of packings corresponding to dually QMDS codes.}  
Theorem~\ref{iffDQMDS} provides a precise geometric characterization of dually QMDS codes in terms of the associated packing: a faithful additive QMDS code is dually QMDS if and only if the intersection of any $|J| \leq k-1$ of its defining subspaces $W_j$ has dimension exactly $r - |J|h$. This implies that the subspaces form a highly regular and rigid structure. It would be interesting to study these packings from a geometric and combinatorial perspective: how rare are they? Do they admit a classification or symmetry group? Can one construct new families of dually QMDS codes by enforcing this intersection condition directly?
\end{enumerate}

\section*{Acknowledgements}
The authors thank the Italian National Group for Algebraic and Geometric Structures and their Applications (GNSAGA—INdAM)
which supported the research. 

Yue Zhou is partially supported by the National Natural Science Foundation of China (No.\ 12371337) and the Natural Science Foundation of Hunan Province (No.\ 2023RC1003)

Alessandro Giannoni thanks the KIMSI Institute in Changsha, China, for hosting him in their facility during the writing of this paper.
\section*{Declarations}
{\bf Conflicts of interest.} The authors have no conflicts of interest to declare that are relevant to the content of this
article.

\end{document}